\documentclass[oneside]{article}
\usepackage{amsmath}
\usepackage{amssymb}
\usepackage{mathrsfs}
\usepackage{geometry}
\usepackage{hyperref}
\usepackage{color}

\numberwithin{equation}{section}

\newtheorem{thm}{Theorem}[section]

\newtheorem{prop}[thm]{Proposition}

\newtheorem{lemma}[thm]{Lemma}

\newtheorem{remark}[thm]{Remark}
\newtheorem{proposition}[thm]{Proposition}
\newtheorem{definition}[thm]{Definition}
\newtheorem{corollary}[thm]{Corollary}

\newenvironment{proof}{{\bf Proof:}}{\hfill$\square$\vskip.5cm}
\newenvironment{proofof}{}{\hfill$\square$\vskip.5cm}
\newcommand{\R}{\mathbf{R}}

\newcommand{\Z}{\mathbf{Z}}
\newcommand{\N}{\mathbf{N}}

\newcommand{\inv}{\operatorname{inv}}

\newcommand{\unity}{\mathbf{1}}

\title{The Length of the Longest Increasing Subsequence of a Random Mallows Permutation}
\author{Carl Mueller$^1$ and Shannon Starr$^2$}

\date{March 24, 2011}

\begin{document}

\markright{}

\maketitle

\footnotetext[1]{This research was supported in part by U.S.\ National Science Foundation
grant DMS-0703855.}
\footnotetext[2]{This research was supported in part by U.S.\ National Science Foundation
grants DMS-0706927 and DMS-0703855.}

\begin{abstract}

The Mallows measure on the symmetric group $S_n$ is the probability measure 
such that each permutation has probability proportional to $q$ raised to the 
power of the number of inversions, where $q$ is a positive parameter and the 
number of inversions of $\pi$ is equal to the number of pairs $i<j$ such that 
$\pi_i > \pi_j$.  We prove a weak law of large numbers for the length of the 
longest increasing subsequence for Mallows distributed random permutations, in 
the limit that $n \to \infty$ and $q \to 1$ in such a way that $n(1-q)$ has a 
limit in $\R$.

\vspace{8pt}
\vskip .2 cm
\noindent
{\small \bf MCS numbers: 335}
\end{abstract}

\section{Main Result}
\label{sec:MainResult}

There is an extensive literature dealing with the longest increasing 
subsequence of a random permutation.  Most of these papers deal with
uniform random permutations.  Our goal is to study the longest 
increasing subsequence under a different measure, the Mallows measure, which 
is motivated by statistics \cite{Mallows}. We begin by defining our terms and stating
the main result, and then we give some historical perspective. 

The $\textrm{Mallows}(n,q)$ probability measure on permutations $S_n$ is given by
\begin{equation}
\label{eq:DefMallowsMeasure}
\mu_{n,q}(\{\pi \})\, =\, [Z(n,q)]^{-1} q^{\inv(\pi)}\, ,
\end{equation}
where $\textrm{inv}(\pi)$ is the number of ``inversions'' of $\pi$,
\begin{equation}
\inv(\pi)\, =\, \#\{ (i,j) \in \{1,\dots,n\}^2\, :\, i<j\, ,\ \pi_i>\pi_j \}\, .
\end{equation}
The normalization is $Z(n,q) = \sum_{\pi \in S_n} q^{\inv(\pi)}$.
See \cite{DiaconisRam} for more background and interesting features of the Mallows measure.
The measure is related to representations of the Iwahori-Hecke algebra as Diaconis
and Ram explain.
It is also related to a natural $q$-deformation of exchangeability which has been
recently discovered and explained by Gnedin and Olshanski \cite{GnedinOlshanski,GnedinOlshanski2}.

We are interested in the length of the longest increasing subsequence
in this distribution.
The length of the longest increasing
subsequence of a permutation $\pi \in S_n$ is 
\begin{equation}
\label{eq:DefLengthLIS}
\ell(\pi)\, =\, \max \{k\leq n\, :\, \pi_{i_1} < \dots < \pi_{i_k}\
\textrm{ for some }\ i_1<\dots<i_k \} \, .
\end{equation}

Our main result is the following.
\begin{thm}
\label{thm:main}
Suppose that $(q_n)_{n=1}^{\infty}$ is a sequence such that the limit
$\beta = \lim_{n \to \infty} n(1-q_n)$ exists.
Then
$$
\lim_{n \to \infty} 
\mu_{n,q_n}\left ( \left \{ \pi \in S_n \, : \, 
|n^{-1/2} \ell(\pi) - \mathcal{L}(\beta)| < \epsilon \right \} \right )\, 
=\, 1\, ,
$$
for all $\epsilon>0$, where
\begin{equation}
\label{eq:LDefinition}
\mathcal{L}(\beta)\, =\, 
\begin{cases} 2 \beta^{-1/2} \sinh^{-1}(\sqrt{e^{\beta}-1}) & \text{ for $\beta>0$,}\\
2 & \text{ for $\beta=0$,}\\
2 |\beta|^{-1/2} \sin^{-1}(\sqrt{1-e^{\beta}})& \text{ for $\beta<0$.}
\end{cases}
\end{equation}
\end{thm}

\begin{remark}
There is another formula for the function in \ref{eq:LDefinition} as
an integral in equation \ref{eq:integralformula}, below.
\end{remark}

In a recent paper \cite{BorodinDiaconisFulman} Borodin, Diaconis and Fulman asked about the Mallows measure,
``Picking a permutation randomly from $P_{\theta}$ (their notation for the Mallows measure),
what is the distribution of the cycle structure, longest increasing subsequence, . . . ?"
We answer the question about the longest increasing subsequence at the level of the weak law of large numbers.

Note that the Mallows measure for $q=1$ reduces to the uniform measure on $S_n$:
\begin{equation}
\notag
\mu_{n,1}(\pi)\, =\, \frac{1}{n!}\, ,
\end{equation}
for all $\pi \in S_n$.
For the uniform measure,  Vershik and Kerov \cite{VershikKerov} and Logan and Shepp \cite{LoganShepp}
already proved a weak law of large numbers for the length of the longest increasing subsequence.
We will use their result in our proof, so we state it here:
\begin{proposition}
\label{prop:VKLSPermutation}
\begin{equation}
\label{eq:PropositionVKLS}
\lim_{n \to \infty} \mu_{n,1}\{ \pi \in S_n\, :\, |n^{-1/2} \ell(\pi) - 2 | > \epsilon \} \, 
=\, 0\, ,
\end{equation}
for all $\epsilon>0$.
\end{proposition}

The reader can find the proof of this proposition in \cite{VershikKerov} and 
\cite{LoganShepp}. 
The proof of Vershik, Kerov, Logan and Shepp involved a deep connection in combinatorics
known as the Robinson-Schensted-Knuth algorithm.
In the RSK algorithm, for each permutation $\pi \in S_n$,
one associates a pair of Young tableaux, with the same shape:
a Young diagram or partition, $\lambda \vdash n$.
Referring to this as $\lambda = \operatorname{sh}(\pi)$,
the first row, $\lambda_1$ is the length of the longest increasing subsequence of $\pi$.
Vershik and Kerov considered the Plancherel measure on the set of partitions, where 
$\mu_n(\lambda)$ is proportional to the number of permutations $\pi$ such that
$\operatorname{sh}(\pi) = \lambda$.
Using the hook length formula of Frame, Robinson and Thrall for this probability measure, 
they defined the ``hook integral''
for limiting (rescaled) shapes of Young tableaux.
They showed concentration of measure of the Plancherel measure,
in the $N \to \infty$ limit, around the optimal shape of the hook integral.
See, for example, \cite{Kerov} for many more details.

The RSK approach was taken further by Baik, Deift and Johansson 
\cite{BaikDeiftJohansson}.
In their seminal work, they gave
a complete description of the fluctuations of $\lambda_1$, which are similar to fluctuations of the largest
eigenvalue in random matrix theory \cite{TracyWidom}.

For the Mallows measure, we do not believe that there is a straightforward analogue of the RSK
algorithm.
The reason for this is that the shape of the Young tableaux in the RSK algorithm does not determine
the number of inversions.
For progress on this difficult problem, see a result of Hohlweg \cite{Hohlweg}.

On the other hand, long after the original proof of Vershik, Kerov, Logan and Shepp,
there was a return to the original probabilistic approach of Hammersley \cite{Hammersley}.
Hammersley had originally asked about the length of the longest increasing subsequence,
as an example of an open problem with partial results, delived in an address directed primarily
to graduate students.
Using methods from interacting particle processes, Aldous and Diaconis
established a hydrodynamic limit of Hammersley's process, thereby giving an independent proof of Theorem \ref{prop:VKLSPermutation}.
This approach was also developed by Seppalainen \cite{Seppalainen}, Cator and Groeneboom \cite{CatorGroeneboom},
and others.
This allows different generalizations than the original RSK proof of Vershik and Kerov.
For instance, Deuschel and Zeitouni considered the length of the longest increasing subsequence
for IID random points in the plane \cite{DeuschelZeitouni,DeuschelZeitouni2}.
For us, their results are key.

On the other hand if, contrary to our expectations, there is a generalization of the RSK algorithm
which applies to Mallows distributed random permutations, that would allow the discoverer
to try to extend the powerful methods of Baik, Deift and Johansson.

The rest of the paper is devoted to the proof of Theorem \ref{thm:main}.
We begin by stating the key ideas.
This occupies Sections 2 through 6.
Certain important technical assumptions will be stated as lemmas.
These lemmas are independent of the main argument,
although the main argument relies on the lemmas.
The lemmas will be proved in Sections 7 and 8.

\section{A Boltzmann-Gibbs measure}
\label{sec:BoltzmannGibbs}

In a previous paper \cite{Starr} one of us proved the following result.
\begin{prop}
\label{lem:StarrOLD}
Suppose that the sequence $(q_n)_{n=1}^{\infty}$ has the limit $\beta = \lim_{n \to \infty} n (1-q_n)$.
For $n \in \N$, let $\pi(\omega) \in S_n$ be a $\textrm{Mallows}(n,q_n)$ random permutation.
For each $n \in \N$, consider the empirical measure 
$\tilde{\rho}_n(\cdot,\omega)$ on $\R^2$, such that
\begin{equation}
\notag
\tilde{\rho}_n(A,\omega)\, =\, \frac{1}{n} \sum_{k=1}^{n} \unity \left\{ 
\left(\frac{k}{n} , \frac{\pi_k(\omega)}{n}\right) \in A \right \}\, ,
\end{equation}
for each Borel set $A \subseteq \R^2$.  Note that 
$\tilde{\rho}_n(\cdot,\omega)$ is a random measure.  
Define the non-random measure ${\rho}_{\beta}$ on $\R^2$ by the formula
\begin{equation}
\label{eq:uDefin}
d{\rho}_{\beta}(x,y)\, =\, 
\frac{(\beta/2) \sinh(\beta/2) \unity_{[0,1]^2}(x,y)}{\big(e^{\beta/4} \cosh(\beta[x-y]/2) - e^{-\beta/4} \cosh(\beta[x+y-1]/2)\big)^2}\, dx\, dy\, .
\end{equation}
Then the sequence of random measures $\tilde{\rho}_n(\cdot,\omega)$ converges in distribution to the non-random
measure ${\rho}_{\beta}$, as $n \to \infty$,
where the convergence is in distribution, relative to the weak topology on Borel probability measures.
\end{prop}

We will reformulate Lemma \ref{lem:StarrOLD}, using a Boltzmann-Gibbs measure
for a classical spin system.  The underlying spins take values in $\R^2$.
We define a two body Hamiltonian interaction $h : \R^2 \to \R$ as
$$
h(x,y)\, =\, \unity\{xy<0\}\, .
$$
Then the $n$ particle Hamiltonian function is $H_n : (\R^2)^n \to \R$,
$$
H_n((x_1,y_1),\dots,(x_n,y_n))\, =\, \frac{1}{n-1} \sum_{i=1}^{n-1} \sum_{j=i+1}^{n} h(x_i-x_j,y_i-y_j)\, .
$$
One also needs an {\em a priori} measure $\alpha$ which is a Borel probability measure
on $\R^2$.
Given all this, the Boltzmann-Gibbs measure on $(\R^2)^n$ with ``inverse-temperature''
$\beta \in \R$ is defined as 
$\mu_{n,\alpha,\beta}$,
$$
d\mu_{n,\alpha,\beta}((x_1,y_1),\dots,(x_n,y_n))\,
=\, \frac{\exp\big(-\beta H_n((x_1,y_1),\dots,(x_n,y_n))\big)\, \prod_{i=1}^{n} d\alpha(x_i,y_i)}
{Z_n(\alpha,\beta)}
$$
where the normalization, known as the ``partition function'' is 
$$
Z_n(\alpha,\beta)\, =\, 
\int_{(\R^2)^n} \exp\big(-\beta H_n((x_1,y_1),\dots,(x_n,y_n))\big)\, \prod_{i=1}^{n} d\alpha(x_i,y_i)\, .
$$
Usually in statistical physics one only considers positive temperatures,
corresponding to $\beta\geq 0$.
But we will also consider $\beta \leq 0$, because it makes mathematical sense and is an interesting
parameter range to study.

A special situation arises when the {\em a priori} measure 
$\alpha$ is a product measure of two one-dimensional measures without atoms.
If $\lambda$ and $\kappa$ are Borel probability measures on $\R$ without atoms,
then
\begin{equation}
\label{eq:equiv}
\begin{split}
\mu_{n,\lambda\times \kappa,\beta}\Big\{((x_1,y_1),\dots,(x_n,y_n)) \in &(\R^2)^n\, : \, 
\exists i_1<\dots<i_k\, ,\\
&\hspace{-5mm} \text{ such that }\
(x_{i_j}-x_{i_\ell})(y_{i_j} - y_{i_\ell}) > 0\ \textrm{ for all $j\neq \ell$} \Big\} \\
&\hspace{5mm}
=\, \mu_{n,\exp(-\beta/(n-1))}(\{ \pi \in S_n \, : \, \ell(\pi) \geq k \})\, ,
\end{split}
\end{equation}
for each $k$.
This follows from the definitions.
In particular, the condition for an increasing subsequence of a permutation $i_1<\dots<i_k$
is that if $i_j < i_\ell$ then we must have $\pi_{i_j} < \pi_{i_\ell}$.
For the variables $(x_1,y_1),\dots,(x_n,y_n)$ replacing the permutation, we obtain
the condition listed above. 

We will also use results from \cite{DeuschelZeitouni} by Deuschel and Zeitouni.
They define the record length of $n$ points in $\R^2$ as 
\begin{equation}
\label{eq:RecordDefinition}
\ell((x_1,y_1),\dots,(x_n,y_n))\, =\, \max \{ k \, : \, \exists i_1<\dots<i_k\, ,\
(x_{i_j}-x_{i_\ell})(y_{i_j} - y_{i_\ell}) > 0\ \textrm{ for all $j<\ell$ } \}\, .
\end{equation}
Equation (\ref{eq:equiv}) says that the distribution of $\ell((X_1(\omega),Y_1(\omega),\dots,X_n(\omega),Y_n(\omega)))$
with respect to the Boltzmann-Gibbs measure $\mu_{n,\lambda \times \kappa,\beta}$
is equal to the distribution of $\ell(\pi(\omega))$ with respect to the 
$\operatorname{Mallows}(n,\exp(-\beta/(n-1)))$ measure $\mu_{n,\exp(-\beta/(n-1))}$.

Using the equivalence and Lemma \ref{lem:StarrOLD}, we may also deduce a weak convergence result
for the measures $\mu_{n,\lambda\times \kappa,\beta}$.
In fact there is a special choice of measure for $\lambda$ and $\kappa$,
depending on $\beta$, which makes the limit nice.

For each $\beta \in \R \setminus \{0\}$ define
$$
L(\beta)\, =\, [(1-e^{-\beta})/\beta]^{1/2}\, ,
$$
and define $L(0)=1$.
Define the Borel probability $\lambda_{\beta}$ on $\R$ by the formula
\begin{equation}
\notag
d\lambda_{\beta}(x)\, =\, \frac{L(\beta) \unity_{[0,L(\beta)]}(x)}{1-\beta L(\beta) x}\, dx\, ,
\end{equation}
for $\beta\neq 0$, and $d\lambda_0(x) = \unity_{[0,1]}(x)\, dx$.
Also define a measure $\sigma_{\beta}$ on $\R^2$ by the formula
$$
d\sigma_{\beta}(x,y)\, =\, \frac{\unity_{[0,L(\beta)]^2}(x,y)}{(1-\beta x y)^2}\, dx\, dy\, .
$$
Both the $x$ and $y$ marginals of 
$\sigma_{\beta}$ are equal to the one-dimensional measure $\lambda_{\beta}$.
Using this, the next lemma follows from Lemma \ref{lem:StarrOLD} 
and the strong law of large numbers.  In fact, the strong law implies that an 
empirical measure arising from i.i.d. samples always converges in distribution 
to the underlying measure, relative to the weak topology on measures.
\begin{lemma}
\label{lem:StarrLimit}
For $n \in \N$,
let $((X_{n,1}(\omega),Y_{n,k}(\omega)),\dots,(X_{n,n}(\omega),Y_{n,n}(\omega)))$ be distributed according to 
the Boltzmann-Gibbs measure
$\mu_{n,\lambda_{\beta}\times \lambda_{\beta},\beta}$, where we used the special {\em a priori}
measure just constructed.
Define
the random empirical measure $\tilde{\sigma}_n(\cdot,\omega)$ on $\R^2$, such that
$$
\tilde{\sigma}_{n}(A,\omega)\, :=\, \frac{1}{n} \sum_{i=1}^{n} \unity \{ (X_{n,i}(\omega),Y_{n,i}(\omega)) \in A \} \, ,
$$
for each Borel measurable set $A \subseteq \R^2$.
Then the sequence of random measures $(\tilde{\sigma}_n(\cdot,\omega))_{n=1}^{\infty}$ converges in distribution to the non-random measure $\sigma_{\beta}$, in the limit $n \to \infty$,
where the convergence in distribution is relative to the topology of weak convergence
of Borel probability measures.
\end{lemma}

We could have also chosen a different {\em a priori} measure to obtain convergence to the same measure
$\rho_{\beta}$ from Lemma \ref{lem:StarrOLD}.
But we find the new measure $\sigma_{\beta}$ to be a nicer parametrization.
We may re-parametrize the measures like this by changing the {\em a priori} measure.
The ability to re-parametrize the measures will also be useful later.

\section{Deuschel and Zeitouni's record lengths}

In  \cite{DeuschelZeitouni},
Deuschel and Zeitouni proved the following result.  We thank Janko Gravner for 
bringing this result to our attention.  
\begin{thm}[Deuschel and Zeitouni, 1995]
\label{thm:DeuschelZeitouni}
Suppose that $u$ is a density on the box $[a_1,a_2]\times [b_1,b_2]$,
i.e., $d\alpha(x,y) = u(x,y)\, dx\, dy$ is a probability measure on the box $[a_1,a_2]\times [b_1,b_2]$.
Also suppose that $u$ is differentiable in $(a_1,a_2)\times (b_1,b_2)$ and the derivative is continuous up to the boundary.
Finally, suppose there exists a constant $c>0$ such that
$$
u(x,y) \geq c\, ,
$$
for all $(x,y) \in [a_1,a_2] \times [b_1,b_2]$.
Let $(U_1,V_1),(U_2,V_2),\dots$ be i.i.d., $\alpha$-distributed
random vectors in $[a_1,a_2]\times [b_1,b_2]$.
Then the rescaled random record lengths,
\begin{equation}
n^{-1/2} \ell((U_1,V_1),\dots,(U_n,V_n))\, ,
\end{equation}
converge in distribution to a non-random number $\mathcal{J}^*(u)$ defined as follows.
Let $\mathcal{C}^1_{\nearrow}([a_1,a_2]\times [b_1,b_2])$ be the set of all
$\mathcal{C}^1$ curves from $(a_1,b_1)$ to $(a_2,b_2)$ whose tangent
line has positive (and finite) slope at all points.
For $\gamma \in \mathcal{C}^1_{\nearrow}([a_1,a_2]\times [b_1,b_2])$ and any
$\mathcal{C}^1$ parametrization $(x(t),y(t))$, define
\begin{equation}
\label{eq:JDefin}
\mathcal{J}(u,\gamma)\, =\, 
2 \int_{\gamma} \sqrt{u(x(t),y(t))\, x'(t)\, y'(t)}\, dt\, .
\end{equation}
This is parametrization independent.
Then
$$
\mathcal{J}^*(u)\, 
=\, \sup_{\gamma \in \mathcal{C}^1_{\nearrow}([a_1,a_2]\times [b_1,b_2])}
\mathcal{J}(u,\gamma)\, .
$$
\end{thm}
This is Theorem 2 in Deuschel and Zeitouni's paper.
The fact that $\mathcal{J}(u,\gamma)$ is parametrization independent
is useful.

We generalize their definition of $\mathcal{J}(u,\gamma)$ a bit, attempting to
mimic the definition of entropy made by Robinson and Ruelle in \cite{RobinsonRuelle}.
This is useful for establishing continuity properties of $\mathcal{J}$ and it allows us to drop the assumption
that $u$ is differentiable.

Given a box $[a_1,a_2]\times [b_1,b_2]$, 
we define $\Pi_n([a_1,a_2]\times [b_1,b_2])$ to be the set of all $(n+1)$-tuples
$\mathcal{P} = ((x_0,y_0),\dots,(x_n,y_n)) \in (\R^2)^{n+1}$ satisfying
$$
a_1 = x_0 \leq \dots \leq x_n = a_2 \quad \textrm{ and } \quad 
b_1 = y_0 \leq \dots \leq y_n = b_2 \, .
$$
We define
\begin{equation}
\label{eq:Jpart}
\tilde{\mathcal{J}}(u,\mathcal{P})\, =\, 2 \sum_{k=0}^{n-1} \left(\int_{x_k}^{x_{k+1}}
\int_{y_k}^{y_{k+1}} u(x,y)\, dx\, dy\right)^{1/2}\, .
\end{equation}
For later reference, we note the following continuity property of $\tilde{\mathcal{J}}(u,\mathcal{P})$  as a function of $u$
for a fixed $\mathcal{P}$.
Suppose that $u$ and $v$ are nonnegative functions in 
$\mathcal{C}([a_1,a_2] \times [b_1,b_2])$.
Using the simple fact that $|a - b| \leq \sqrt{|a^2-b^2|}$,
for all $a,b\geq 0$, we see that
$$
|\tilde{\mathcal{J}}(u,\mathcal{P}) - \tilde{\mathcal{J}}(v,\mathcal{P})| \,
\leq \, 2 \sum_{k=0}^{n-1} \left(\int_{x_k}^{x_{k+1}}
\int_{y_k}^{y_{k+1}} |u(x,y)-v(x,y)|\, dx\, dy\right)^{1/2}\, .
$$
We define $\|u\|$ to be the supremum norm.
Using this and the Cauchy inequality,
\begin{equation}
\label{eq:Holder}
\begin{split}
|\tilde{\mathcal{J}}(u,\mathcal{P}) - \tilde{\mathcal{J}}(v,\mathcal{P})|\,
&\leq \, 2 \|u-v\|^{1/2}\, \sum_{k=0}^{n-1} \sqrt{(x_{k+1}-x_k)(y_{k+1}-y_k)}\\
&\leq \,  \|u-v\|^{1/2}\, \sum_{k=0}^{n-1} (x_{k+1}-x_k+y_{k+1}-y_k)\\
&=\,  \|u-v\|^{1/2}\, (a_2-a_1+b_2-b_1)\, .
\end{split}
\end{equation}
Now we state a technical lemma.

\begin{lemma}
\label{lem:USC}
Let 
$\mathcal{B}_{\nearrow}([a_1,a_2]\times [b_1,b_2])$ be the set of all 
connected sets $\Upsilon \subset [a_1,a_2]\times [b_1,b_2]$ containing $(a_1,b_1)$ and $(a_2,b_2)$, and having the property
that 
$(x_1-x_2)(y_1-y_2) \geq 0$
for all $(x_1,y_1), (x_2,y_2) \in \Upsilon$.
Define $\Pi_n(\Upsilon)$ to be the set of all 
$\mathcal{P} = ((x_0,y_0),\dots,(x_n,y_n))$ in $\Pi_n$ such that 
$(x_k,y_k) \in \Upsilon$ for each $k$,
and let $\Pi(\Upsilon) = \bigcup_{n=1}^{\infty} \Pi_n(\Upsilon)$.
Finally, define
$$
\tilde{\mathcal{J}}(u,\Upsilon)\, =\, \lim_{\epsilon \to 0}\,
\inf \left \{
\tilde{\mathcal{J}}(u,\mathcal{P}) \,
:\, \mathcal{P} \in \bigcup_{n=1}^{\infty} \Pi_n(\Upsilon)\, ,\ \|\mathcal{P}\|< \epsilon
\right \}
\, .
$$
Then $\tilde{\mathcal{J}}(u,\cdot)$ is an
upper semi-continuous function of 
$\mathcal{B}_{\nearrow}([a_1,a_2]\times [b_1,b_2])$, endowed with the Hausdorff metric.
\end{lemma}
If $\Upsilon$ is the range of a 
curve $\gamma \in \mathcal{C}^1_{\nearrow}([a_1,a_2]\times [b_1,b_2])$,
then
$\tilde{\mathcal{J}}(u,\Upsilon) = \mathcal{J}(u,\gamma)$
because for each partition $\mathcal{P} \in \Pi(\Upsilon)$,
the quantity $\tilde{\mathcal{J}}(u,\mathcal{P})$ just gives a
Riemann sum approximation to the integral in $\mathcal{J}(u,\gamma)$.

Now, let us denote the density of $\sigma_{\beta}$ as
\begin{equation}
\label{eq:Definitionubeta}
u_{\beta}(x,y)\, =\, \frac{\unity_{[0,L(\beta)]^2]}(x,y)}{(1-\beta xy)^2}\, .
\end{equation}
Then we may prove the following variational calculation.
\begin{lemma}
\label{lem:variational}
For any $\Upsilon \in \mathcal{B}_{\nearrow}([0,L(\beta)]^2)$,
$$
\tilde{\mathcal{J}}(u_{\beta},\Upsilon)\,
\leq\, \int_0^{L(\beta)} \frac{2}{1-\beta t^2}\, dt\, =\, \mathcal{L}(\beta)\, .
$$
\end{lemma}
Let us quickly verify the lemma in the special case $\beta=0$.
We have set $L(0)=1$ and we know that $u_0$ is identically 1 on
the rectangle $[0,1]^2$.
By equation (\ref{eq:Holder}), we know that
$$
\tilde{\mathcal{J}}(u_0,\mathcal{P})\, \leq\, 2\, ,
$$
by comparing $u=u_0$ with $v=0$.
That means that $\tilde{\mathcal{J}}(u_0,\Upsilon) \leq 2$ for every choice of $\Upsilon$.
It is easy to see that taking $\Upsilon = \{(t,t)\, :\, 0\leq t\leq 1\}$,
which is the graph of the straight line curve $\gamma$ parametrized by $x(t)=y(t)=t$
for $0\leq t\leq 1$,
$$
\tilde{\mathcal{J}}(u_0,\Upsilon)\, =\, \mathcal{J}(u_0,\gamma)\, =\, 2 \int_0^1 
\sqrt{u_0(x(t),y(t)) x'(t) y'(t)}\, dt\, =\, 2\, .
$$
Therefore, using Deuschel and Zeitouni's theorem, this shows that the straight
line is the optimal path for the case of a constant density on a square.

This lemma in general is proved using basic inequalities, as above, 
combined with the fact that $\mathcal{J}(u,\gamma)$
is parametrization independent, which allows us to reparametrize time
for any curve $(x(t),y(t))$.
As with the other lemmas, we prove this in Section 7 at the end of the paper.


\section{Coupling to IID point processes}

Now, suppose that $\beta$ is fixed, and consider a triangular array of random vectors in $\R^2$,
$$
((X_{n,k},Y_{n,k})\, :\, n \in \N\, ,\ 1\leq k\leq n)\\, ,
$$ 
where for each $n \in \N$, the 
random variables $(X_{n,1},Y_{n,1}),\dots,(X_{n,n},Y_{n,n})$ are distributed
according to the Boltzmann-Gibbs measure $\mu_{n,\lambda_{\beta}\times \lambda_{\beta},\beta}$.
We know that 
$$
\mu_{n,\exp(-\beta/(n-1))}\{\ell(\pi)=k\}\, =\,
{\bf P}
\{\ell((X_{n,1},Y_{n,1}),\dots,(X_{n,n},Y_{n,n})) = k\}\, ,
$$
for each $k$.
We also know that the empirical measure associated to $((X_{n,1},Y_{n,1}),\dots,(X_{n,n},Y_{n,n}))$
converges to the special measure $\sigma_{\beta}$.
It is natural to try to apply Deuschel and Zeitouni's Theorem 
\ref{thm:DeuschelZeitouni}, even though the points  $(X_{n,1},Y_{n,1}),\dots,(X_{n,n},Y_{n,n})$
are not i.i.d., a requirement for the random variables 
$(U_1,V_1),\dots,(U_n,V_n)$ of their theorem.

It is useful to generalize our perspective slightly.
Let us suppose that $\lambda$ and $\kappa$ are general Borel probability measures on $\R$
without atoms, and let us consider a triangular array of random vectors in $\R^2$:
$((X_{n,k},Y_{n,k})\, :\, n \in \N\, ,\ 1\leq k\leq n)$, where for each $n \in \N$, the 
random variables $(X_{n,1}(\omega),Y_{n,1}(\omega)),\dots,(X_{n,n}(\omega),Y_{n,n}(\omega))$ are distributed
according to the Boltzmann-Gibbs measure $\mu_{n,\lambda\times \kappa,\beta}$.
Let us define the random non-normalized, integer valued Borel measure $\xi_n(\cdot,\omega)$ on $\R^2$, by
\begin{equation}
\label{eq:Definitionxin}
\xi_n(A,\omega)\, =\, \sum_{i=1}^{n} \unity \{ (X_{n,i}(\omega),Y_{n,i}(\omega)) \in A \}\, ,
\end{equation}
This is a {\em random point process}.

A general point process is a random, locally finite, nonnegative integer valued measure.
We will restrict attention to finite point processes.
Therefore, let $\mathcal{X}$ denote the set of all Borel measures $\xi$ on $\R^2$
such that $\xi(A) \in \{0,1,\dots \}$ for each Borel measurable set $A \subseteq \R^2$.
Then, almost surely, $\xi_n(\cdot,\omega)$ is in $\mathcal{X}$.
In fact $\xi_n(\R^2,\omega)$ is a.s.\ just $n$.
For a general random point process, the total number of points may be random.
\begin{definition}
\label{def:nu}
Let $\nu_{n,\lambda\times \kappa,\beta}$ be the Borel probability measure on $\mathcal{X}$
describing the distribution of the random element $\xi_n(\cdot,\omega) \in \mathcal{X}$
defined in (\ref{eq:Definitionxin}), where $(X_{n,1}(\omega),Y_{n,1}(\omega)),\dots,(X_{n,n}(\omega),Y_{n,n}(\omega))$ are distributed
according to the Boltzmann-Gibbs measure $\mu_{n,\lambda\times \kappa,\beta}$.
\end{definition}
Given a measure $\xi \in \mathcal{X}$, we extend the definition of the record length to
\begin{equation}
\label{eq:DefinitionRecordPointProcess}
\begin{split}
\ell(\xi)\, =\, \max \{ k \, &: \, \exists (x_1,y_1),\dots,(x_k,y_k) \in \R^2\ \text { such that }\\
& \xi(\{(x_1,y_1),\dots,(x_k,y_k)\}) \geq k\ \text{ and }
(x_i-x_j)(y_i-y_j) \geq 0 \text{ for all } i,j \}\, .
\end{split}
\end{equation}
With this definition, 
\begin{equation}
\label{eq:LBoldL}
\ell(\xi_n(\cdot,\omega))\, =\, \ell((X_{n,1}(\omega),Y_{n,1}(\omega)),\dots,(X_{n,n}(\omega),Y_{n,n}(\omega)))\, ,
\end{equation}
almost surely.

There is a natural order on measures.
If $\mu,\nu$ are two measures on $\R^2$, then let us say
$\mu \leq \nu$ if $\mu(A) \leq \nu(A)$, for each Borel set $A \subseteq \R^2$.
The function $\ell$ is monotone non-decreasing in the sense that if $\xi,\zeta$
are two measures in $\mathcal{X}$ then $\xi \leq \zeta \Rightarrow \ell(\xi) \leq \ell(\zeta)$.

\begin{lemma}
\label{lem:CouplingAboveBelow}
Suppose that $\lambda$ and $\kappa$ each have no atoms. Then for each $n \in \N$, the following holds.
\begin{itemize}
\item[(a)] There exists a pair of random point processes $\eta_n,\xi_n$, defined on the same probability space, such that $\eta_n \leq \xi_n$, a.s., and satisfying these conditions:
$\xi_n$ has distribution $\nu_{n,\lambda \times \kappa,\beta}$; there are i.i.d., Bernoulli-$p$ random variables $K_1,\dots,K_n$, for $p=\exp(-|\beta|)$,
and i.i.d., $\lambda \times \kappa$-distributed points $(U_1,V_1),\dots,(U_{K_1+\dots+K_n},V_{K_1+\dots+K_n})$, such that
$\eta_n(A) = \sum_{i=1}^{K_1+\dots+K_n} \unity \{ (U_i,V_i) \in A \}$.
\item[(b)] There exists a pair of random point processes $\xi_n,\zeta_n$, defined on the same probability space, such that $\xi_n \leq \zeta_n$, a.s., and satisfying these conditions: $\xi_n$ has distribution $\nu_{n,\lambda \times \kappa,\beta}$; there are i.i.d., geometric-$p$ random variables $N_1,\dots,N_n$, for $p=\exp(-|\beta|)$,
and i.i.d., $\lambda \times \kappa$-distributed points $(U_1,V_1),\dots,(U_{N_1+\dots+N_n},V_{N_1+\dots+N_n})$, such that
$\zeta_n(A) = \sum_{i=1}^{N_1+\dots+N_n} \unity \{ (U_i,V_i) \in A \}$.
\end{itemize}
\end{lemma}

We may combine this lemma with the weak law of large numbers and the Vershik and Kerov, Logan and Shepp theorem, to conclude the following:
\begin{corollary}
\label{cor:ProbabilityBounds}
Suppose that $(q_n)_{n=1}^{\infty}$ is a sequence such that $\lim_{n \to \infty} n(1-q_n)  = \beta \in \R$.
Then,
\begin{equation}
\notag
\lim_{n \to \infty} \mu_{n,q_n} \{ \pi \in S_n \, : \, 
n^{-1/2} \ell(\pi) \in (2 e^{-|\beta|/2} - \epsilon,2e^{|\beta|/2}+\epsilon) \}\, 
=\, 1\, ,
\end{equation}
for each $\epsilon>0$.
\end{corollary}

Let us quickly prove this corollary, conditional on previously stated lemmas whose proofs
will appear later.

\begin{proofof}{\bf Proof of Corollary \ref{cor:ProbabilityBounds}:}
Let $\beta_n$ be defined so that $\exp(-\beta_n/(n-1)) = q_n$.
Let $\pi \in S_n$ be a random permutation, distributed according to $\mu_{n,q_n}$,
and let $((X_{n,1},Y_{n,1}),\dots,(X_{n,n},Y_{n,n}))$
be distributed according to $\mu_{n,\lambda\times \kappa,\beta_n}$.
We have the equality in distribution of the random variables
$$
\ell((X_{n,1},Y_{n,1}),\dots,(X_{n,n},Y_{n,n}))\, 
\stackrel{\mathcal{D}}{=} \, \ell(\pi)\, ,
$$
as we noted in Section 2, before.
Note $\lim_{n \to \infty} n (1-q_n) = \beta$,
implies that $\lim_{n \to \infty} \beta_n = \beta$.

For a fixed $n$, we apply Lemma \ref{lem:CouplingAboveBelow}, but with $\beta$ replaced by $\beta_n$, to conclude that 
there are random point processes
$\eta_n(\cdot,\omega), \xi_n(\cdot,\omega) \in \mathcal{X}$ defined on the same probability space $\Omega$,
and separately, there are random point processes $\xi_n(\cdot,\omega), \zeta_n(\cdot,\omega) \in \mathcal{X}$, 
defined on the same probability space, satisfying the conclusions of that lemma but with $\beta$ replaced by $\beta_n$.
By (\ref{eq:LBoldL}), we know that
$$
\ell(\pi)\, \stackrel{\mathcal{D}}{=} \, \ell(\xi)\, .
$$
By monotonicity of $\ell$, and Lemma \ref{lem:CouplingAboveBelow} we know that for each $k$
\begin{equation}
\label{eq:CouplingAboveBelowk}
{\bf P}\{\ell(\eta) \geq k\}\, 
\leq\, {\bf P}\{\ell(\xi) \geq k\} 
\quad \text { and } \quad
{\bf P}\{\ell(\xi) \geq k\} \,
\leq \, {\bf P}\{\ell(\zeta) \geq k\}\, .
\end{equation}
Using equations (\ref{eq:equiv}) and  (\ref{eq:LBoldL}), this implies that for each $\epsilon>0$
\begin{gather*}
\mu_{n,q_n}\{ \pi \in S_n\, :\, n^{-1/2} \ell(\pi) \leq 2 e^{|\beta|/2}+\epsilon \}\,
\geq\, {\bf P}\{n^{-1/2} \ell(\zeta) \leq 2 e^{|\beta|/2}+\epsilon \}\, ,\\
\mu_{n,q_n}\{ \pi \in S_n\, :\, n^{-1/2} \ell(\pi) \geq 2 e^{-|\beta|/2}-\epsilon \}\,
\geq\, {\bf P}\{n^{-1/2} \ell(\eta) \geq 2 e^{-|\beta|/2}-\epsilon \}\, .
\end{gather*}
Since the $(U_i,V_i)$'s end at $i=K_1+\dots+K_n$ or $i=N_1+\dots+N_n$ in the two cases,
let us also define new i.i.d., $\lambda\times\kappa$-distributed points $(U_i,V_i)$ for all greater values of $i$.
We assume these are independent of everything else.
Then all $(U_i,V_i)$ are i.i.d., $\lambda\times \kappa$ distributed.
So, for any non-random number $m \in \N$,
the induced permutation $\pi_m \in S_m$, corresponding to $((U_1,V_1),\dots,(U_m,V_m))$
is uniformly distributed.

The random integers $K_1,\dots,K_n$ and $N_1,\dots,N_n$ from Lemma \ref{lem:CouplingAboveBelow} are not independent of $(U_1,V_1),(U_2,V_2),\dots$.
But, for instance, for any deterministic number $m$, conditioning on the event $\{\omega \in \Omega\, :\, K_1(\omega) + \dots + K_n(\omega) \leq m\}$, we have that
$$
\ell(\zeta)\, \leq\, \ell\big((U_1,V_1),\dots,(U_{m},V_{m})\big)\, ,
$$
by using monotonicity of $\ell$ again.
Therefore, for each $n \in \N$, and for any non-random number $M_n^+ \in \N$, we may bound
\begin{align*}
{\bf P}\{n^{-1/2} \ell(\zeta) \leq 2 e^{|\beta|/2}+\epsilon \}\, 
&\geq\, \mu_{M_n^+,1}\{ \pi \in S_{M_n^+}\, :\, n^{-1/2} \ell(\pi) \leq 2 e^{|\beta|/2}+\epsilon \}\\
&\qquad - {\bf P}(\{\omega\in \Omega\, :\, K_1(\omega)+\dots+K_n(\omega) > M_n^+\})\, .
\end{align*}
Similarly, for any non-random number $M_n^-$, we may bound
\begin{align*}
{\bf P}\{n^{-1/2} \ell(\zeta) \geq 2 e^{-|\beta|/2}-\epsilon \}\, 
&\geq\, \mu_{M_n^-,1}\{ \pi \in S_{M_n^-}\, :\, n^{-1/2} \ell(\pi) \geq 2 e^{-|\beta|/2}-\epsilon \}\\
&\qquad - {\bf P}(\{\omega \in \Omega\, :\, N_1(\omega)+\dots+N_n(\omega) < M_n^-\})\, .
\end{align*}
We choose $\delta$ such that $0<\delta<\epsilon$, and 
then we take sequences $M_n^+ = \lfloor n (e^{-|\beta|}+\delta) \rfloor$ and
$N_n^- = \lceil n (e^{|\beta|} - \delta) \rceil$.
Since  $K_1,K_2,\dots$ are i.i.d., Bernoulli random variables
with mean $e^{-|\beta|}$,
and $N_1,N_2,\dots$ are i.i.d., geometric random variables with mean $e^{|\beta|}$,
we may appeal to the weak law of large numbers to deduce
$$
\lim_{n \to \infty} {\bf P}(\{\omega\in \Omega\, :\, K_1(\omega)+\dots+K_n(\omega) > M_n^+\})\,
=\, \lim_{n \to \infty} {\bf P}(\{\omega \in \Omega\, :\, N_1(\omega)+\dots+N_n(\omega) < M_n^-\})\,
=\, 0\, .
$$

Finally,
by Proposition \ref{prop:VKLSPermutation}, we know that
\begin{multline*}
\liminf_{n \to \infty}  \mu_{M_n^+,1}\{ \pi \in S_{M_n^+}\, :\, n^{-1/2} \ell(\pi) \leq 2 e^{|\beta|/2}+\epsilon \}\\
\geq\, \liminf_{n \to \infty} \mu_{M_n^+,1}\left\{ \pi \in S_{M_n^+}\, :\, 
(M_n^+)^{-1/2} \ell(\pi) \leq 2 \frac{e^{|\beta|/2}+\epsilon}{e^{|\beta|/2}+\delta}\right\}
=\, 1\, ,
\end{multline*}
and
\begin{multline*}
\liminf_{n \to \infty}  \mu_{M_n^-,1}\{ \pi \in S_{M_n^-}\, :\, n^{-1/2} \ell(\pi) \geq 2 e^{-|\beta|/2}-\epsilon \}\\
\geq\, \liminf_{n \to \infty} \mu_{M_n^-,1}\left\{ \pi \in S_{M_n^-}\, :\, 
(M_n^-)^{-1/2} \ell(\pi) \geq 2 \frac{e^{-|\beta|/2}-\epsilon}{e^{-|\beta|/2}-\delta}\right\}
=\, 1\, .
\end{multline*}
\end{proofof}

The bounds in  Corollary \ref{cor:ProbabilityBounds} are useful for small values of $|\beta|$.
For larger values of $\beta$, they are useful when combined with the following easy lemma:
\begin{lemma}
\label{lem:TemperatureRenormalization}
Suppose $\lambda$ and $\kappa$ have no atoms,
and let the random point process $\xi \in \mathcal{X}$ be distributed according to $\nu_{n,\lambda\times \kappa,\beta}$.
Suppose that $R = [a_1,a_2] \times [b_1,b_2]$ is any rectangle.
Let $\xi \restriction R$ denote the restriction of $\xi$ to this rectangle:
i.e., $(\xi \restriction R)(A) = \xi(A \cap R)$.
Note that this is still a random point process in $\mathcal{X}$ but one with a random
total mass between $0$ and $n$.
Then, for any $m \in \{1,\dots,n \}$, and any $k \in \{1,\dots,m\}$, we have
\begin{equation}
{\bf P}(\{\ell(\xi \restriction R) = k\}\, |\, \{\xi(R)=m\})\,
=\, \mu_{m,q}\{ \pi \in S_m\, :\, \ell(\pi) = k\}\, ,
\end{equation}
for $q = \exp(-\beta/(m-1))$.
\end{lemma}

In order to use this lemma, we introduce an idea we call ``paths of boxes.''

\section{Paths of boxes}

We now introduce a method to derive Deuschel and Zeitouni's Theorem \ref{thm:DeuschelZeitouni}
for our point process.
For each $n$ we decompose the unit square $[0,1]^2$ into $n^2$ sub-boxes
$$
R_n(i,j)\, =\, \left[\frac{i-1}{n},\frac{i}{n}\right] \times \left[\frac{j-1}{n},\frac{j}{n}\right]\, .
$$
We consider a basic path to be a sequence $(i_1,j_1),\dots,(i_{2n-1},j_{2n-1})$ such that
$(i_1,j_1)=(1,1)$, $(i_{2n-1},j_{2n-1}) = (n,n)$ and $(i_{k+1}-i_k,j_{k+1}-j_k)$ equals $(1,0)$ or $(0,1)$
for each $k=1,\dots,2n-2$.
In this case the basic path of boxes is the union $\bigcup_{k=1}^{2n-1} R_n(i_k,j_k)$.
Note that 
\begin{gather*}
(i_{k+1}-i_k,j_{k+1}-j_k)\, =\, (1,0)\quad \Rightarrow\quad 
R_n(i_k,j_k) \cap R_n(i_{k+1},j_{k+1})\, =\, \{i_k/n\} \times [(j_k-1)/n,j_k/n]\, ,\\
(i_{k+1}-i_k,j_{k+1}-j_k)\, =\, (0,1)\quad \Rightarrow\quad 
R_n(i_k,j_k) \cap R_n(i_{k+1},j_{k+1})\,  =\,[(i_k-1)/n,i_k/n] \cap \{j_k/n\}\, .
\end{gather*}
 
Now we consider a refined notion of path.
We are motivated by the fact that Deuschel and Zeitouni's $\mathcal{J}(u,\gamma)$ function does depend on the derivative of $\gamma$.
To get reasonable error bounds we must allow for a choice of slope for each segment of the path.
So, given $m \in \N$ and $n \in \{2,3,\dots \}$, we consider a set of ``refined'' paths $\Pi_{n,m}$ to be the set of all sequences
$$
\Gamma\, :=\, ((i_1,j_1),r_1,(i_2,j_2),r_2,(i_3,j_3),r_3,\dots,(i_{2n-2},j_{2n-2}),r_{2n-2},(i_{2n-1},j_{2n-1}))\, ,
$$
where $((i_1,j_1),(i_2,j_2),\dots,(i_{2n-1},j_{2n-1})$ is a basic path, as described in the last paragraph, and $r_1,r_2,\dots,r_{2n-2}$
are integers in $\{1,\dots,m\}$
satisfying the additional condition: if $i_{k} = i_{k+1} = i_{k+2}$ or if $j_k=j_{k+1}=j_{k+2}$ then $r_{k+1} \geq r_k$,
for each $k=1,\dots,2n-3$.
We now explain the importance of this condition.

Suppose that $R_n(i_k,j_k) \cap R_n(i_{k+1},j_{k+1}) = \{i_k/n\} \times [(j_k-1)/n,j_{k}/n]$.
Then we decompose this interval into $m$ subintervals 
$$
I_{n,m}^{(2)}(i_k;j_k,j_{k+1};r)\, =\, \left \{\frac{i_k}{n}\right\} \times \left[ \frac{j_k-1}{n} + \frac{r-1}{mn}, \frac{j_k-1}{n} + \frac{r}{m} \right]\, .
$$
Similarly, if $R_n(i_k,j_k) \cap R_n(i_{k+1},j_{k+1}) = [(i_k-1)/n,i_{k}/n] \times \{j_k/n\}$, then we define
$$
I_{n,m}^{(1)}(i_k,i_{k+1};j_k;r)\, =\, \left[ \frac{i_k-1}{n} + \frac{r-1}{m}, \frac{i_k-1}{n} + \frac{r}{mn} \right] \times 
\left \{\frac{j_k}{n}\right\}\, .
$$
In either case, the choice of $r_k$ is which subinterval the ``path'' passes through in going from $R_n(i_k,j_k)$ to $R_n(i_{k+1},j_{k+1})$.
We define $I_k$ to be $I_{n,m}^{(2)}(i_k;j_k,j_{k+1};r_k)$ or $I_{n,m}^{(1)}(i_k,i_{k+1};j_k;r_k)$ depending on which case it is.
We also define $(x_k,y_k)$ to be the center of the interval, either
$$
(x_k,y_k)\, =\, \left(\frac{i_k}{n},\frac{j_k-1}{n} + \frac{r-(1/2)}{mn}\right)\quad \textrm{or}\quad 
(x_k,y_k)\, =\, \left(\frac{i_k-1}{n} + \frac{r-(1/2)}{mn},\frac{j_k}{n}\right)\, .
$$
The additional condition that we require for a refined path just guarantees that $x_{k+1} \geq x_k$ and $y_{k+1} \geq y_k$ for each $k$.

We also define $(a_k,b_k) \in \R^2$ and $(c_k,d_k) \in \R^2$ to be the endpoints of the interval $I_k$.
With these definitions, we may state our main result for paths of boxes.

\begin{lemma}
\label{lem:PathOfBoxes}
Suppose that $\Gamma \in \Pi_{n,m}$ is a refined path.
Also suppose that $\xi \in \mathcal{X}$ is a point process with support in $[0,1]^2$,
such that no point lies on any line $\{(x,y)\, :\, x = i/n \}$ for $i \in \Z$
or any line $\{(x,y)\, :\, y=j/n \}$ for $j \in \Z$. 
Then
$$
\ell(\xi)\, \geq\, \sum_{k=1}^{2n-1} \ell(\xi \restriction [x_{k-1},x_k]\times [y_{k-1},y_k])\, ,
$$
where we define $(x_0,y_0)=(0,0)$ and $(x_{2n-1},y_{2n-1})=(1,1)$.
Also,
$$
\ell(\xi)\, \leq\, \max_{\Gamma \in \Pi_{n,m}} \sum_{k=1}^{2n-1} \ell(\xi \restriction [a_{k-1},c_k] \times [b_{k-1},d_k])\, ,
$$
where we define $(a_0,b_0)=(0,0)$ and $(c_{2n-1},d_{2n-1}) = (1,1)$.
\end{lemma}
We will prove this lemma in Section 8, after we have proved the other lemmas,
since it requires several steps.

Another useful lemma follows:

\begin{lemma}
\label{lem:boxes}
Suppose that $u:[0,1]^2 \to \R$ is a probability density which is also continuous.
Then, 
$$
\max_{\Upsilon \in\mathcal{B}_{\nearrow}([0,1]^2)} 
\tilde{\mathcal{J}}(u,\Upsilon)\, 
=\, 2 \lim_{N \to \infty}\, \lim_{m \to \infty}\, \max_{\Gamma \in \Pi_{n,m}} \sum_{k=1}^{2N-1} 
\left(\int_{x_{k-1}}^{x_k} \int_{y_{k-1}}^{y_k} u(x,y)\, dx\, dy\right)^{1/2}\, .
$$
\end{lemma}

We will prove this simple lemma in Section 7.
With these preliminaries done, we may now complete the proof of the theorem.

\section{Completion of the Proof}

Suppose that $\beta \in \R$ is fixed.
At first we will consider a fixed sequence $q_n = \exp(-\beta/(n-1))$, which does satisfy $n(1-q_n) \to \beta$ as $n \to \infty$.
Define the triangular array of random vectors in $\R^2$:
$((X_{n,k},Y_{n,k})\, :\, n \in \N\, ,\ 1\leq k\leq n)$, where for each $n \in \N$, the 
random variables $(X_{n,1},Y_{n,1}),\dots,(X_{n,n},Y_{n,n})$ are distributed
according to the Boltzmann-Gibbs measure $\mu_{n,\lambda\times \kappa,\beta}$.
Let $\xi_n \in \mathcal{X}$ be the random point process such that
$$
\xi_n(A)\, =\, \sum_{k=1}^{n} \unity \{ (X_{n,k},Y_{n,k}) \in A \} \, ,
$$
for each Borel measurable set $A \subseteq \R^2$.
As we have noted before, we then have
\begin{eqnarray}
\notag
\mu_{n,q_n}\{ \pi \in S_n \, : \, \ell(\pi) = k \} \,
&=&\, P\{ \ell((X_{n,1},Y_{n,1}),\dots,(X_{n,n},Y_{n,n})) = k \} \\
\notag
&=&\, P\{ \ell(\xi_n) = k \}\, ,
\end{eqnarray}
for each $k$.

Now suppose that $m, N \in \N$ are fixed.
We consider ``refined'' paths in $\Pi_{N,m}$.
By Lemma \ref{lem:PathOfBoxes}, which applies by first rescaling the unit square $[0,1]^2$ to $[0,L(\beta)]^2$,
\begin{equation}
\label{eq:lower}
\ell(\xi_n)\, \geq \, \max_{\Gamma \in \Pi_{N,m}} \sum_{k=1}^{2N-1} \ell(\xi_n \restriction [L(\beta) x_{k-1},L(\beta) x_k]\times [L(\beta) y_{k-1},L(\beta) y_k])\, .
\end{equation}
The only difference is that we use the square $[0,L(\beta)]^2$ in place of $[0,1]^2$.
Also, 
\begin{equation}
\label{eq:upper}
\ell(\xi_n)\, \leq \, \max_{\Gamma \in \Pi_{N,m}} \sum_{k=1}^{2N-1} \ell(\xi_n \restriction [L(\beta) a_{k-1},L(\beta) c_k] 
\times [L(\beta) b_{k-1}, L(\beta) d_k])\, .
\end{equation}
Now suppose that $\Gamma \in \Pi_{N,m}$ is fixed.
Also consider a fixed sub-rectangle of $\Gamma$,
$$
R_k\, =\, [L(\beta) x_{k-1},L(\beta) x_k]\times [L(\beta) y_{k-1},L(\beta) y_k]\, .
$$
By Lemma \ref{lem:StarrLimit}, we know that the random variables
$\xi_n(R_k)/n$ converge in probability to the non-random limit $\sigma_{\beta}(R_k)$,
as $n \to \infty$.
Moreover, conditioning on the total number of points in the sub-rectangle $\xi_n(R_k)$,
Lemma \ref{lem:TemperatureRenormalization} tells us that
$$
{\bf P}( \{\ell(\xi_n \restriction R_k) = \bullet\}\, |\, \{\xi_n(R_k) = r\})\, =\,
\mu_{r,q_n}\{\pi \in S_r\, :\, \ell(\pi) = \bullet \}\, .
$$
Note that the sequence of random variables $\xi_n(R_k) (1 - q_n)$ converges in probability
to $\beta \sigma_{\beta}(R_k)$ as $n\to \infty$, because
$$
\xi_n(R_k) (1 - q_n)\, =\, n(1-q_n)\, \frac{\xi_n(R_k)}{n}\, ,
$$
and $n(1-q_n) \to \beta$ as $n \to \infty$.
Therefore, using Corollary \ref{cor:ProbabilityBounds}, this implies for each $\epsilon>0$
$$
\lim_{n \to \infty} {\bf P}\left\{\xi_n(R_k)^{-1/2} \ell(\xi_n \restriction R)
\in (2 e^{-\beta \sigma_{\beta}(R_k)/2} - \epsilon, 2 e^{\beta \sigma_{\beta}(R_k)/2} + \epsilon)
\right\} \,
=\, 1\, .
$$
Since we have a limit in probability for $\xi_n(R_k)/n$, we may then conclude for each $\epsilon>0$ that
$$
\lim_{n \to \infty} {\bf P}\left\{ n^{-1/2} \ell(\xi_n \restriction R_k)
\in (2 [\sigma_{\beta}(R_k)]^{1/2} e^{-\beta \sigma_{\beta}(R_k)/2} - \epsilon, 
2 [\sigma_{\beta}(R_k)]^{1/2} e^{\beta \sigma_{\beta}(R_k)/2} + \epsilon)
\right) \,
=\, 1\, .
$$

This is true for each sub-rectangle $R_k$ comprising $\Gamma$, and $\Gamma$ is in $\Pi_{N,m}$.
But there are only finitely many sub-rectangles in $\Gamma$, and there are only finitely many
possible choices of a refined path of boxes $\Gamma \in \Pi_{N,m}$, for $N$ and $m$ fixed.
Combining this with (\ref{eq:lower}) implies that for any $\epsilon>0$ we have
\begin{equation}
\label{eq:FinalLowerBound}
\lim_{n \to \infty} 
{\bf P}\left\{ n^{-1/2} \ell(\xi_n) \geq \max_{\Gamma \in \Pi_{m,n}}
\sum_{k=1}^{2N-1} 2 [\sigma_{\beta}(R_k)]^{1/2} e^{-\beta \sigma_{\beta}(R_k)/2} - \epsilon 
\right\}\, =\, 1\, .
\end{equation}
By exactly similar arguments and (\ref{eq:upper}) we may also conclude that for each $\epsilon>0$
\begin{equation}
\label{eq:FinalUpperBound}
\lim_{n \to \infty} 
{\bf P}\left\{  n^{-1/2} \ell(\xi_n) \leq \max_{\Gamma \in \Pi_{m,n}}
\sum_{k=1}^{2N-1} 2 [\sigma_{\beta}(R^*_k)]^{1/2} e^{\beta \sigma_{\beta}(R^*_k)/2} + \epsilon 
\right\}\, =\, 1\, ,
\end{equation}
where we define 
$$
R_k^*\, =\, [L(\beta) a_{k-1},L(\beta) c_k],[L(\beta) b_{k-1},L(\beta) d_k]\, ,
$$
for each $k = 1,\dots,2N-1$.

We apply Lemma \ref{lem:boxes} to $u_{\beta}$.
For $N $ fixed, taking the limit $m\to \infty$, the  area of the symmetric differences of the boxes $R_k^*$ and $R_k$
converges to zero, uniformly in $\Gamma \in \Pi_{N,m}$ for each $k=1,\dots,2N-1$.
Since $\sigma_{\beta}$ has a density, the same is true replacing area by $\sigma_{\beta}$-measure.
Moreover, $\exp(-\beta \sigma_{\beta}(R_k))$ and $\exp(\beta \sigma_{\beta}(R_k^*))$ converge
to 1 uniformly as $N \to \infty$.
Therefore,
\begin{equation}
\label{eq:continuity}
\begin{split}
&\lim_{N \to \infty} \lim_{m \to \infty} \max_{\Gamma \in \Pi_{m,n}}
\sum_{k=1}^{2N-1} 2 [\sigma_{\beta}(R_k)]^{1/2} e^{-\beta \sigma_{\beta}(R_k)/2}\\
&\hspace{1cm}
=\, \lim_{N \to \infty} \lim_{m \to \infty} \max_{\Gamma \in \Pi_{m,n}}
\sum_{k=1}^{2N-1} 2 [\sigma_{\beta}(R_k^*)]^{1/2} e^{-\beta \sigma_{\beta}(R_k^*)/2}\\
&\hspace{2cm} = \max_{\Upsilon \in \mathcal{B}_{\nearrow}([0,L(\beta)]^2)} \tilde{\mathcal{J}}(u_{\beta},\Upsilon)\, .
\end{split}
\end{equation}
Combined with (\ref{eq:FinalLowerBound}) and (\ref{eq:FinalUpperBound}),
this implies that for each $\epsilon>0$,
$$
\lim_{n \to \infty} {\bf P} \left\{\left| n^{-1/2} \ell(\xi_n)
- \max_{\Upsilon \in \mathcal{B}_{\nearrow}([0,L(\beta)]^2} \tilde{\mathcal{J}}(u_{\beta},\Upsilon)\right| < \epsilon \right\}\, =\, 1\, .
$$
Finally, 
we use Lemma \ref{lem:variational} to conclude that 
$$
\max_{\Upsilon \in \mathcal{B}_{\nearrow}([0,L(\beta)]^2)} \tilde{\mathcal{J}}(u_{\beta},\Upsilon)\,
\leq\, \mathcal{L}(\beta)\, .
$$
But taking $\Upsilon = \{(t,t)\, :\, t \in [0,L(\beta)]\}$, which is the graph of the straight line curve $\gamma \in \mathcal{C}^1_{\nearrow}([0,L(\beta)]^2)$, gives 
$$
\tilde{\mathcal{J}}(u_{\beta},\Upsilon)\, 
=\, \mathcal{J}(u_{\beta},\gamma)\, =\, 2 \int_0^{L(\beta)} \frac{1}{1-\beta t^2}\, dt\, .
$$
This integral gives $\mathcal{L}(\beta)$.

Thus, the proof is completed, for the special choice of $(q_n)$
equal to $(\exp(-\beta/(n-1)))$.
Because the answer is continuous in $\beta$,  if we consider any sequence
$(q_n)$ satisfying $n (1-q_n) \to \beta$, then we get the same answer.
All that is left is to prove all the lemmas.

\section{Proofs of  Lemmas \ref{lem:USC}, \ref{lem:variational}, \ref{lem:TemperatureRenormalization} and \ref{lem:boxes} }

We now prove the lemmas, in an order which is not necessarily the same as the order they were stated.
This facilitates using arguments from one proof for the next one.

\begin{proofof}{\bf Proof of Lemma \ref{lem:USC}.}
Define
$$
\tilde{\mathcal{J}}_{\epsilon}(u,\Upsilon)\, 
=\, \inf \{ \tilde{\mathcal{J}}(u,\mathcal{P})\, :\, \mathcal{P} \in \Pi(\Upsilon)\, ,\
\|\mathcal{P}\| < \epsilon \}
$$
for each $\epsilon>0$.
We first show that this function is upper semi-continuous.

Let $\Pi_n$ denote $\Pi_n([a_1,a_2]\times [b_1,b_2])$.
We remind the reader that this is the set of all $(n+1)$-tuples $\mathcal{P} = ((x_0,y_0),\dots,(x_n,y_n)) \in (\R^2)^{n+1}$ such that
$a_1=x_0\leq \dots \leq x_n=a_n$ and $b_1=y_0\leq \dots \leq y_n=b_2$.
For each $\mathcal{P} \in \Pi_n$, we have
$$
\tilde{\mathcal{J}}(u,\mathcal{P})\,
=\, \sum_{k=0}^{n-1} 
\left(\int_{x_k}^{x_{k+1}} \int_{y_k}^{y_{k+1}} u(x,y)\, dx\, dy\right)^{1/2}\, .
$$
Since $u$ is continuous, the mapping $\tilde{\mathcal{J}}(u,\cdot) : \Pi_n \to \R$
is continuous when $\Pi_n$ has its usual topology as a subset of
$(\R^2)^{n+1}$.

Consider a fixed path $\Upsilon \in \mathcal{B}_{\nearrow}([a_1,a_2] \times [b_1,b_2])$
and a partition $\mathcal{P} \in \Pi(\Upsilon)$ such that $\|\mathcal{P}\| < \epsilon$.
Note that there is some $n$ such that $\mathcal{P} \in \Pi_n(\Upsilon)$.
Suppose that $(\Upsilon^{(k)})_{k=1}^{\infty}$ is a sequence in 
$\mathcal{B}_{\nearrow}([a_1,a_2] \times [b_1,b_2])$
converging to $\Upsilon$ in the Hausdorff metric.
Then for each point $(x,y) \in \Upsilon$, there is a sequence of points 
$(x^{(k)},y^{(k)}) \in \Upsilon^{(k)}$ converging to $(x,y)$.
Therefore, we may choose a sequence of partitions 
$\mathcal{P}^{(k)} \in \Pi_n(\Upsilon^{(k)})$
converging to $\mathcal{P}$ in $\Pi_n$.
By the continuity mentioned above,
$$
\lim_{k \to \infty} \tilde{\mathcal{J}}(u,\mathcal{P}^{(k)}) \to 
\tilde{\mathcal{J}}(u,\mathcal{P})\, .
$$
Also, $\|\mathcal{P}^{(k)}\|$ converges to $\|\mathcal{P}\|$ which is  less than $\epsilon$. So, for large enough $k$, we have $\|\mathcal{P}^{(k)}\| < \epsilon$, and hence
$$
\tilde{\mathcal{J}}(u,\mathcal{P}^{(k)})\, \geq \,
\tilde{\mathcal{J}}_{\epsilon}(u,\Upsilon^{(k)})\, ,
$$
since the right hand side is the infimum.
Therefore, we see that
$$
\limsup_{k \to \infty}\, \tilde{\mathcal{J}}_{\epsilon}(u,\Upsilon^{(k)})\,
\leq \, \tilde{\mathcal{J}}(u,\mathcal{P}) \, .
$$
Since this is true for all $\mathcal{P} \in \Pi(\Upsilon)$ with $\| \mathcal{P} \| < \epsilon$,
taking the infimum we obtain
$$
\limsup_{k \to \infty}\, \tilde{\mathcal{J}}_{\epsilon}(u,\Upsilon^{(k)})\,
\leq \, \tilde{\mathcal{J}}_{\epsilon}(u,\Upsilon)\, .
$$
Since this is true for every $\Upsilon \in \mathcal{B}_{\nearrow}([a_1,a_2] \times [b_1,b_2])$ and every sequence  
$(\Upsilon^{(k)})$
converging to $\Upsilon$ in the Hausdorff metric, this proves that
$\tilde{\mathcal{J}}_{\epsilon}(u,\cdot)$ is upper semi-continuous
on $\mathcal{B}_{\nearrow}([a_1,a_2] \times [b_1,b_2])$.
\end{proofof}

\begin{proofof}{\bf Proof of Lemma \ref{lem:boxes}:}
The proof of this lemma is also used in the proof of Lemma \ref{lem:variational}.
This is the reason it appears first.

Recall the definition of the basic boxes for $i,j \in \{1,\dots,N\}$,
$$
R_N(i,j)\, =\, \left[\frac{i-1}{N},\frac{i}{N}\right] \times \left[\frac{j-1}{N},\frac{j}{N}\right]\, .
$$
Given $N \in \N$, let us define $u_N^+$ and $u_N^-$ so that
\begin{gather*}
u_N^+(x,y)\, =\, \sum_{i,j=1}^{N} \max_{(x',y') \in R_N(i,j)}  u(x',y') \cdot
{\bf 1}_{R_N(i,j)}(x,y)\,  ,\\
u_N^-(x,y)\, =\, \sum_{i,j=1}^{N} \min_{(x',y') \in R_N(i,j)}  u(x',y') \cdot
{\bf 1}_{R_N(i,j)}(x,y)\, .
\end{gather*}
By monotonicity, $\mathcal{J}(u_N^-,\Upsilon) \leq \mathcal{J}(u,\Upsilon) \leq \mathcal{J}(u_N^+,\Upsilon)$
for every $\Upsilon \in \mathcal{B}_{\nearrow}([0,1]^2)$.
But since $u_N^-$ and $u_N^+$ are constant on squares, we know that the optimal $\Upsilon$'s for $u_N^-$ and $u_N^+$
are graphs of rectifiable curves $\gamma$ that are piecewise straight line curves on squares.
This follows from the discussion immediately following the statement of Lemma \ref{lem:variational},
where we verified the special case of that lemma for constant densities.
The only degrees of freedom for such curves are the slopes of each straight line, i.e., where they intersect the boundaries of each basic square.

For $(x_k,y_k), (x_{k+1},y_{k+1}) \in R_N(i,j)$ representing two points on the boundary, such that $x_{k-1}\leq x_k$ and $y_{k-1}\leq y_k$,
considering $\gamma_k$ to be the straight line
joining these points,
$$
\int_{\gamma_k} \sqrt{u_N^+(x(t),y(t)) x'(t) y'(t)}\, dt\, =\, \sqrt{(x_k-x_{k-1})(y_k-y_{k-1})} \max_{(x,y) \in R_N(i,j)} \sqrt{u(x,y)}\, ,
$$
with a similar formula for $u^-$.
This is a continuous function of the endpoints.
We may approximate the actual optimal piecewise straight line path by the "refined paths" of boxes in $\Pi_{N,m}$ if we take
the limit $m \to \infty$  with $N$ fixed.
Therefore, we find that
$$
\max_{\Upsilon \in \mathcal{B}_{\nearrow}([0,1]^2)} \tilde{\mathcal{J}}(u^\pm_N,\Upsilon)\,
=\, \lim_{m \to \infty}\, \max_{\Gamma \in \Pi_{m,n}} \sum_{k=1}^{2N-1} 
\left(\int_{x_{k-1}}^{x_k} \int_{y_{k-1}}^{y_k} u^\pm_N(x,y)\, dx\, dy\right)^{1/2}\, .
$$
Note that by upper semicontinuity, for each fixed $N$, the limit as $m \to \infty$ of the sequence
$$
\max_{\Gamma \in \Pi_{m,n}} \sum_{k=1}^{2N-1} 
\left(\int_{x_{k-1}}^{x_k} \int_{y_{k-1}}^{y_k} u(x,y)\, dx\, dy\right)^{1/2}
$$
also exists, and is the supremum over $m \in \N$.
Therefore, we conclude that for each fixed $N \in \N$,
\begin{align*}
&\hspace{-1cm}
\lim_{m \to \infty}\, \max_{\Gamma \in \Pi_{m,n}} \sum_{k=1}^{2N-1} 
\left(\int_{x_{k-1}}^{x_k} \int_{y_{k-1}}^{y_k} u(x,y)\, dx\, dy\right)^{1/2}\\
&\hspace{1cm} \in \left[\max_{\Upsilon \in \mathcal{B}_{\nearrow}([0,1]^2)} \tilde{\mathcal{J}}(u^-_N,\Upsilon)\, ,\
\max_{\Upsilon \in \mathcal{B}_{\nearrow}([0,1]^2)} \tilde{\mathcal{J}}(u^+_N,\Upsilon)\right]\, .
\end{align*}
But taking $N \to \infty$, we see that $u_N^+$ and $u_N^-$ converge to $u$, uniformly due to the continuity
of $u$.
Therefore, by the bound from equation (\ref{eq:Holder}), the lemma follows.
\end{proofof}

\begin{proofof}{\bf Proof of Lemma \ref{lem:variational}:}
Suppose that $x(t)$, $y(t)$ is a $\mathcal{C}^1$ parametrization of a curve 
$\gamma \in \mathcal{C}^1_{\nearrow}([0,L(\beta)]^2)$.
We may consider another time parametrization
$x_1(t) = x(f(t))$ and $y_1(t) = y(f(t))$ for a $\mathcal{C}^1$ function $f(t)$
such that
$$
x_1(t) y_1(t)\, =\, t^2\, .
$$
Indeed, we obtain $x(f(t)) y(f(t)) = t^2$. Setting $g(t) = x(t) y(t)$,
our assumptions on $x(t)$ and $y(t)$ guarantee that $g$ is continuous and 
$g'(t)$ is strictly positive and finite for all $t$.
We then take $f(t) = g^{-1}(t^2)$.

Since a change of time parametrization does not affect $\mathcal{J}(u_{\beta},\gamma)$,
we will simply assume that $x(t) y(t) = t^2$ is satisfied at the outset.
Then we obtain
$$
\mathcal{J}(u_{\beta},\gamma)\, 
=\, \int_0^{L(\beta)} \frac{2 \sqrt{x'(t) y'(t)}}{1-\beta t^2}\, dt\, ,
$$
due to the formula for $u_{\beta}$, and the fact that $x(t) y(t) = t^2 = L^2(\beta)$
at the endpoint of $\gamma$.
Now since we have $x(t) y(t) = t^2$, that implies that
\begin{equation}
\label{eq:constraint}
x(t) y'(t) + y(t) x'(t)\, =\, 2 t\, .
\end{equation}
We know that $x'(t)$ and $y'(t)$ are nonnegative. 
Therefore, we may use Cauchy's inequality with $\epsilon$
$$
\sqrt{x'(t) y'(t)}\, = \, [x'(t)]^{1/2} [y'(t)]^{1/2}\, \leq \,
\frac{\epsilon}{2} x'(t) + \frac{1}{2\epsilon} y'(t)\, ,
$$
for each $\epsilon \in (0,\infty)$.
Taking $\epsilon = y(t)/t$ we get $\epsilon^{-1} = t/y(t)$ which is $x(t)/t$ since
we chose the parametrization that $x(t) y(t) = t^2$.
Therefore, we obtain
$$
\sqrt{x'(t) y'(t)}\, \leq \, \frac{ y(t) x'(t) + x(t) y'(t)}{2t}\, .
$$
Taking into account our constraint (\ref{eq:constraint}), this gives 
$$
\sqrt{x'(t) y'(t)}\, \leq \, 1\, .
$$
Since this is true at all $t \in [0,L(\beta)]$ this proves the desired inequality.
But this upper bound gives the integral 
\begin{equation}
\label{eq:integralformula}
\int_0^{L(\beta)} \frac{dt}{1-\beta t^2}\, =\, 
\int_0^{[(1-e^{-\beta})/\beta]^{1/2}} \frac{dt}{1-\beta t^2}\, ,
\end{equation}
which equals the formula for $\mathcal{L}(\beta)$ from (\ref{eq:LDefinition}).

The argument works even if $\gamma$ is only piecewise $\mathcal{C}^1$, with finitely many pieces.
Moreover, by the proof of Lemma \ref{lem:boxes}, we know that the maximum over all $\Upsilon$
is arbitrarily well approximated by optimizing over piecewise linear paths.
So the inequality is true in general.
\end{proofof}

\begin{proofof}{\bf Proof of Lemma \ref{lem:TemperatureRenormalization}:}
This lemma is related to an important independence property of the Mallows measure.
Gnedin and Olshanski prove this in Proposition 3.2 of \cite{GnedinOlshanski2},
and they note that Mallows also stated a version in \cite{Mallows}.
Our lemma is slightly different so we prove it here for completeness.

Using Definition \ref{def:nu}, we can instead consider $(X_{1},Y_{1}),\dots,(X_{n},Y_{n})$
distributed according to $\mu_{n,\lambda\times \kappa,\beta}$ in place of $\xi$ distributed
according to $\nu_{n,\lambda\times \kappa,\beta}$.
Given $m\leq n$, we note that, conditioning on the positions of $(X_{m+1},Y_{m+1}),\dots,(X_n,Y_n)$,
the conditional distribution of $(X_1,Y_1),\dots,(X_m,Y_m)$ is the same as $\mu_{m,\alpha,\beta'}$,
where $\beta' = (m-1)\beta/(n-1)$ and where $\alpha$ is the random measure
$$
d\alpha(x,y)\, =\, \frac{1}{Z_1}\, \exp\left(-\frac{\beta}{n-1} \sum_{i=m+1}^n h(x-X_i,y-Y_i)\right)\,
d\lambda(x)\, d\kappa(y)\, ,
$$
where $Z_1$ is a random normalization constant.
By finite exchangeability of $\mu_{n,\lambda\times \kappa,\beta}$ it does not matter which $m$
points we assume are in the square $[a_1,a_2]\times [b_1,b_2]$ which is why we just chose
the first $m$.

If we could rewrite $\alpha$ as a product of two measures $\lambda',\kappa'$ without atoms
then we could appeal to (\ref{eq:equiv}).
By inspection $\alpha$ is not a product of two measures.
However, if we condition on the event that there are exactly $m$ points in the square $[a_1,a_2]\times [b_1,b_2]$
then we can accomplish this goal.
Let use define the event $A = \{(X_{m+1},Y_{m+1}),\dots,(X_n,Y_n) \not\in [a_1,a_2]\times [b_1,b_2]\}$.
Then, given the event $A$, we can write
\begin{equation}
\label{eq:cond}
{\bf 1}_{[a_1,a_2]\times [b_1,b_2]}(x,y)\, d\alpha(x,y)\, =\,  d\lambda'(x)\, d\kappa'(y)\, ,
\end{equation}
where $\lambda'$ and $\kappa'$ are random measures
$$
d\lambda'(x)\, =\, \frac{1}{Z_2}\, e^{-\beta h_1(x)/(n-1)}\, d\lambda(x)\, ,\quad
d\kappa'(y)\, =\, \frac{1}{Z_3}\, e^{-\beta h_2(y)/(n-1)}\, d\kappa(y)\, ,
$$
with $Z_2$ and $Z_3$ normalization constants and random functions
$$
h_1(x)\, =\, \sum_{i=m+1}^{n} [{\bf 1}_{\{Y_i < b_1\}} {\bf 1}_{(X_i,\infty)}(x) + {\bf 1}_{\{Y_i>b_2\}} {\bf 1}_{(-\infty,X_i)}(x)]\, ,
$$
and
$$
h_2(y)\, =\, \sum_{i=m+1}^{n} [{\bf 1}_{\{X_i < a_1\}} {\bf 1}_{(Y_i,\infty)}(y) + {\bf 1}_{\{X_i>a_2\}} {\bf 1}_{(-\infty,Y_i)}(x)]\, .
$$
This may appear not to reproduce $\alpha$ exactly because it may seem that $h_1$ and $h_2$ double-count some terms
which are only counted once in  $\sum_{i=m+1}^{n} h(x-X_i,y-Y_i)$.
But this is compensated by the normalization constants $Z_1$ and $Z_2$ as we now explain.

Note that for each $i \in \{m+1,\dots,n\}$ since $(X_i,Y_i) \not\in [a_1,a_2] \times [b_1,b_2]$ we either have
$Y_i<b_1$, $Y_i>b_2$, $X_i<a_1$ or $X_i>a_2$.
These are not exclusive.
But for instance, if $Y_i<b_1$ and $X_i<a_1$ then for every $(x,y) \in [a_1,a_2] \times [b_1,b_2]$, we have
${\bf 1}_{\{Y_i < b_1\}} {\bf 1}_{(X_i,\infty)}(x)=1$ and ${\bf 1}_{\{X_i < a_1\}} {\bf 1}_{(Y_i,\infty)}(y)=1$.
Therefore, these terms are constant in the functions $h_1(x)$ and $h_2(y)$: they do not depend on the actual
position of $(x,y)$ as long as $(x,y) \in [a_1,a_2] \times [b_1,b_2]$.
Therefore, using the normalization constants $Z_1$ and $Z_2$, this double-counting may be compensated.

Since we are conditioning on $\{(X_1,Y_1),\dots,(X_m,Y_m) \in [a_1,a_2] \times [b_1,b_2]\}$ and the event $A$,
the conditional identity (\ref{eq:cond}) suffices to prove the claim.
\end{proofof}

\section{ Proof of Lemma \ref{lem:CouplingAboveBelow}}

This is the most involved lemma to prove.
It follows from a coupling argument.
In fact we use the most basic type of coupling for discrete random variables,
based on the total variation distance.
See the monograph \cite{LevinPeresWilmer} (Chapter 4) for a nice and elementary exposition.
But we also combine this with the fact that we have a measure which may
be derived from a statistical mechanical model of {\em mean field} type.
Because the model is of mean field type, the correlations are weak and spread out.
In principle, this allows one to approximate by a mixture of i.i.d., points as one sees in de Finetti's
theorem in probability, or the Kac-Lebowitz-Penrose limit in statistical physics.
(See \cite{Aldous} for a reference on the former, and the appendix of \cite{Thompson} 
for the latter.)

Given a probability measure $\alpha$ on $\R^2$, let $\theta_{1,\alpha}$ be the distribution on $\mathcal{X}$
associated to the random point process
$$
\xi_1(A,\omega)\, =\, {\bf 1}_A(X(\omega),Y(\omega))\, ,
$$
assuming $(X(\omega),Y(\omega))$ is $\alpha$-distributed.

\begin{lemma}
\label{lem:Coupling}
Suppose that $\alpha$ and $\tilde{\alpha}$ are two measures on $\R^2$ such that $\tilde{\alpha} \ll \alpha$, and suppose that
for some $p \in (0,1]$ there are uniform bounds
$$
p\, \leq\, \frac{d\tilde{\alpha}}{d\alpha}\, \leq\, p^{-1}\, .
$$
Then the following holds.
\begin{itemize}
\item[(a)] There exists a pair of random point processes $\eta_1, \xi_1$, defined on the same probability space,
such that $\eta_1 \leq \xi_1$, a.s., and satisfying these properties:
$\xi_1$ has distribution $\theta_{1,\tilde{\alpha}}$;
there is an $\alpha$-distributed random point $(U_1,V_1)$, and independently there is a 
Bernoulli-$p$ random variable $K_1$, such that $\eta_1(A) = K_1 {\bf 1}_A(U_1,V_1)$.
\item[(b)] There exists a pair of random point processes $\xi_1,\zeta_1$, defined on the same probability space,
such that $\xi_1 \leq \zeta_1$, a.s., and satisfying these properties:
$\xi_1$ has distribution $\theta_{1,\tilde{\alpha}}$;
there is a sequence of i.i.d., $\alpha$-distributed points $(U_1,V_1),(U_2,V_2),\dots$ and a
geometric-$p$ random variable $N_1$, such that
$\zeta_1(A) = \sum_{i=1}^{N_1} {\bf 1}_A(U_i,V_i)$.
\end{itemize}
\end{lemma}

\begin{proof}
Let $f = d\tilde{\alpha}/d\alpha$.
We follow the standard approach, for example in Section 4.2 of \cite{LevinPeresWilmer}.
We describe it here in detail, in order to be self-contained.
Define $g(x) = (1-p)^{-1} [f(x)-p]$, which is nonnegative by assumption,
and let $\hat{\alpha}$ be the probability measure such that
$d\hat{\alpha}/d\alpha = g$.
Note that $\tilde{\alpha}$ can be written as a mixture: $\tilde{\alpha} = p \alpha + (1-p) \hat{\alpha}$.

Independently of one another, let $(U_1,V_1) \in \R^2$ be $\alpha$-distributed, and let $(W_1,Z_1) \in \R^2$
be $\hat{\alpha}$-distributed.
Independently of all that, also let $K_1$ be Bernoulli-$q$.
Then, taking
$$
(X_1,Y_1)\, =\, \begin{cases} (U_1,V_1) & \text { if $K_1 = 1$,}\\
(W_1,Z_1) & \text { otherwise,}
\end{cases}
$$
we see that $(X_1,Y_1)$ is $\tilde{\alpha}$-distributed.
We let $\eta_1(A,\omega) = K_1(\omega) {\bf 1}_A(U_1(\omega),V_1(\omega))$.
If $K_1(\omega)=1$ then $(U_1(\omega),V_1(\omega)) = (X_1(\omega),Y_1(\omega))$.
Therefore taking $\xi_1(A,\Omega) = {\bf 1}_A(X_1(\omega),Y_1(\omega))$, we see
that $\eta_1(\cdot,\omega) \leq \xi_1(\cdot,\omega)$, a.s.
This proves (a). 

The proof for (b) is analogous. Let $h(x) = (p^{-1} - 1)^{-1}[p^{-1} - f(x)]$, which is nonnegative by hypothesis.
Let $\check{\alpha}$ be the probability measure such that $d\check{\alpha}/d\alpha = h$.
Then $\alpha$ can be written as the mixture: $\alpha = p \tilde{\alpha} + (1-p) \check{\alpha}$.
Independently of each other, let $(X_1,Y_1),(X_2,Y_2),\dots$ be i.i.d., $\tilde{\alpha}$ distributed random variables,
and let $(Z_1,W_1), (Z_2,W_2),\dots$ be i.i.d., $\check{\alpha}$ distributed random variables.
Also, independently of all that, let $K_1,K_2,\dots$ be i.i.d., Bernoulli-$q$ random variables.
For each $i$, we define
$$
(U_i,V_i)\, =\, \begin{cases} (X_i,Y_i) & \text { if $K_i=1$,}\\
(Z_i,W_i) & \text { otherwise.}
\end{cases}
$$
Then $(U_1,V_1),(U_2,V_2),\dots$ are i.i.d., $\alpha$-distributed random variables.
Let $N_1 = \min\{n\, :\, K_n = 1\}$.
We see that $(X_{N_1},Y_{N_1}) = (U_{N_1},V_{N_1})$.
So clearly ${\bf 1}_A(X_{N_1},Y_{N_1}) \leq \sum_{k=1}^{N_1} {\bf 1}_A(U_k,V_k)$.
\end{proof}

Note that $K_1$ and $N_1$ are random variables which are dependent on  $(U_1,V_1),(U_2,V_2),\dots$.
But, for instance, conditioning on the event $\{N_1\geq i\}$, we do see that $(U_i,V_i)$ is $\alpha$-distributed.
This is for the usual reason, as in Doob's optional stopping theorem:
the event $\{N_1 \geq i\}$ is measurable with respect to the $\sigma$-algebra
generated by $K_1,\dots,K_{i-1}$, while the point $(U_i,V_i)$ is independent of that
$\sigma$-algebra.
This will be useful when we consider $n>1$, which is next.

\subsection{Resampling and Coupling for $n>1$}

In order to complete the proof of Lemma \ref{lem:CouplingAboveBelow} we want to use Lemma \ref{lem:Coupling}.
More precisely we wish to iterate the bound for $n>1$.
Suppose that $\tilde{\alpha}_n$ is a probability measure on $(\R^2)^n$, and $\alpha$ is a probability measure on $\R^2$.
Let $\theta_{n,\tilde{\alpha}_n}$ be the distribution on $\mathcal{X}$
associated to the random point process
$$
\xi_n(A,\omega)\, =\, \sum_{k=1}^{n} {\bf 1}_A(X_k(\omega),Y_k(\omega))\, ,
$$
assuming $(X_1(\omega),Y_1(\omega)),\dots,(X_n(\omega),Y_n(\omega))$ are $\tilde{\alpha}_n$-distributed.

If $\tilde{\alpha}_n$ was a product measure then it would be trivial to generalize Lemma \ref{lem:Coupling}
to compare it to the product measure $\alpha^n$.
But there is another condition which makes it equally easy to generalize.
Let $\mathcal{F}$ denote the Borel $\sigma$-algebra on $\R^2$.
Let $\mathcal{F}^n$ denote the Borel $\sigma$-algebra on $(\R^2)^n$.
Let $\mathcal{F}^n_k$ denote the sub-$\sigma$-algebra of $\mathcal{F}^n$
generated by the maps $((x_1,y_1),\dots,(x_n,y_n)) \mapsto (x_j,y_j)$
for $j \in \{1,\dots,n\} \setminus \{k\}$.
We suppose that there are regular conditional probability measures for each of these sub-$\sigma$-algebras.
Let us 
make this precise:

\begin{definition}
\label{def:rcpd}
We say that $\tilde{\alpha}_{n,k} : \mathcal{F} \times (\R^2)^n \to \R$ is a regular conditional probability measure
for $\tilde{\alpha}_n$, relative to the sub-$\sigma$-algebra $\mathcal{F}^n_k$ if the following three conditions are met:
\begin{enumerate}
\item For each $((x_1,y_1),\dots,(x_n,y_n)) \in (\R^2)^n$
the mapping
$$
A\mapsto \tilde{\alpha}_{n,k}\big(A;(x_1,y_1),\dots,(x_n,y_n)\big)
$$
defines a probability measure on $\mathcal{F}$.
\item For each $A \in \mathcal{F}$, the mapping
$$
((x^1,y^2),\dots,(x^n,y^n)) \mapsto  \tilde{\alpha}_{n,k}\big(A;(x_1,y_1),\dots,(x_n,y_n)\big)
$$
is $\mathcal{F}^n$ measurable.
\item The measure $\tilde{\alpha}_{n,k}$ is a {\em version} of the conditional expectation ${\bf E}^{\tilde{\alpha}_n}[\cdot\, |\, \mathcal{F}^n_k]$.
In this case this means precisely that for each $A_1,\dots,A_n \in \mathcal{F}$,
$$
{\bf E}^{\tilde{\alpha}_n}\Bigg[\tilde{\alpha}_{n,k}\Big(A_k;(X_1,Y_1),\dots,(X_n,Y_n)\Big) \prod_{\substack{j=1\\ j\neq k}}^n {\bf 1}_{A_j}(X_j,Y_j)\Bigg]\,
=\, \tilde{\alpha}_n(A_1\times \cdots \times A_n)\, .
$$
\end{enumerate}
\end{definition}

For $p \in (0,1]$,
we will say that $\tilde{\alpha}_n$ satisfies the $p$-resampling condition relative to $\alpha$ if the following conditions are satisfied:
\begin{itemize}
\item There exist regular conditional probability distributions $\tilde{\alpha}_{n,k}$  relative to $\mathcal{F}^n_k$
for $k=1,\dots,n$.
\item For each $((x_1,y_1),\dots,(x_n,y_n)) \in \R^n$, and for each $k=1,\dots,n$,
$$
\tilde{\alpha}_{n,k}(\cdot;(x_1,y_1),\dots,(x_n,y_n))\, \ll\, \alpha\, .
$$
\item The following uniform bounds are satisfied for each $((x_1,y_1),\dots,(x_n,y_n)) \in \R^n$, and for each $k=1,\dots,n$:
$$
p\, \leq\, \frac{d\tilde{\alpha}_{n,k}(\cdot;(x_1,y_1),\dots,(x_n,y_n))}{d\alpha}\, \leq p^{-1}\, .
$$
\end{itemize}

\begin{lemma}
\label{lem:CouplingGeneral}
Suppose that for some $p \in (0,1]$, the measure $\tilde{\alpha}_n$
satisfies the $p$-resampling condition relative to $\alpha$.
Then the following holds.
\begin{itemize}
\item[(a)] There exists a pair of random point processes $\eta_n, \xi_n$, defined on the same probability space,
such that $\eta_n \leq \xi_n$, a.s., and satisfying these properties:
$\xi_n$ has distribution $\theta_{n,\tilde{\alpha}_n}$;
there are i.i.d., $\alpha$-distributed points $\{(U^k_1,V^k_1)\}_{k=1}^n$, and independently there are i.i.d.,
Bernoulli-$p$ random variables $K_1,\dots,K_n$, such that $\eta_n(A) = \sum_{k=1}^{n} K_k {\bf 1}_A(U^k_1,V^k_1)$.
\item[(b)] There exists a pair of random point processes $\xi_n,\zeta_n$, defined on the same probability space,
such that $\xi_n \leq \zeta_n$, a.s., and satisfying these properties:
$\xi_n$ has distribution $\theta_{n,\tilde{\alpha}_n}$;
there are i.i.d., $\alpha$-distributed points $\{(U^k_i,V^k_i)\, :\, k=1,\dots,n\, ,\ i=1,2,\dots\}$,
and i.i.d., geometric-$p$ random variables $N_1,\dots,N_n$, such that
$\zeta_n(A) = \sum_{k=1}^{n} \sum_{i=1}^{N_k} {\bf 1}_A(U^k_i,V^k_i)$.
\end{itemize}
\end{lemma}

\begin{proof}
We start with an $\tilde{\alpha}_n$-distributed  random point
$((X^k_1,Y^k_1),\dots,(X^k_n,Y^k_n))$.
Then iteratively, for each $k=1,\dots,n$, we update this point as follows.
Conditional on 
$$
((X^{k-1}_1,Y^{k-1}_1),\dots,(X^{k-1}_n,Y^{k-1}_n))\, ,
$$ 
we choose $(X^k_k,Y^k_k)$
randomly, according to the distribution
$$
\tilde{\alpha}_{n,k}\big(\cdot;(X^{k-1}_1,Y^{k-1}_1),\dots,(X^{k-1}_n,Y^{k-1}_n)\big)\, .
$$
We let $(X^k_j,Y^k_j) = (X^{k-1}_j,Y^{k-1}_j)$ for each $j \in \{1,\dots,n\} \setminus \{k\}$.
With this resampling rule, we can see that $((X^k_1,Y^k_1),\dots,(X^k_n,Y^k_n))$
is $\tilde{\alpha}_n$-distributed for each $k$.
Also $(X^n_k,Y^n_k) = (X^k_k,Y^k_k)$.

We apply Lemma \ref{lem:Coupling} to each of the points $(X^k,Y^k)$, in turn.
Since they all have distributions satisfying the hypotheses of the lemma, this may be done.
Note that by our choices, the various $(U^k_i,V^k_i)$'s and $K_k$'s and $N_k$'s have distributions
which are prescribed just in terms of $p$ and $\alpha$.
Their distributions do not depend on the regular conditional probability distributions,
as long as the hypotheses of the present lemma are satisfied.
Therefore, they are independent of one another.
\end{proof}

Given the lemma, for part (a) we let $(U_1,V_1),\dots,(U_{K_1+\dots+K_n},V_{K_1+\dots+K_n})$ be equal to the points $(U^k_1,V^k_1)$ such that $K_k=1$,
suitably relabeled, but keeping the relative order.
By the idea, related to Doob's stopping theorem, that we mentioned before, one can see that 
$$
(U_1,V_1),\dots,(U_{K_1+\dots+K_n},V_{K_1+\dots+K_n})
$$
are i.i.d.,
$\alpha$-distributed.
We do similarly in case (b).
This allows us to match up our notation with Lemma \ref{lem:CouplingAboveBelow}.
The only thing left is to check that ``$p$-resampling condition'' for the regular conditional probability distributions
is satisfied for Boltzmann-Gibbs distributions.

\subsection{Regular conditional probability distributions for the Boltzmann-Gibbs measure}

In Lemma \ref{lem:CouplingAboveBelow}, we assume that $((X_1,Y_1),\dots,(X_n,Y_n))$ are distributed according
to the Boltzmann-Gibbs measure $\mu_{n,\lambda \times \kappa,\beta}$.
Then we let $\nu_{n,\lambda\times \kappa,\beta}$ be the distribution of the random point process
$\xi_n$, such that
$$
\xi_n(A)\, =\, \sum_{k=1}^{n} {\bf 1}_A(X_k,Y_k)\, .
$$
In other words, the distribution $\mu_{n,\lambda\times\kappa,\beta}$ corresponds to the distribution
we have denoted $\theta_{n,\tilde{\alpha}_n}$ if we let $\tilde{\alpha}_n = \mu_{n,\lambda \times \kappa,\beta}$.
We take $\alpha = \lambda \times \kappa$.
Now we want to verify the hypotheses of Lemma \ref{lem:CouplingGeneral} for $p=e^{-|\beta|}$.

Referring back to Section \ref{sec:BoltzmannGibbs}, we see that $\tilde{\alpha}_n$ is absolutely continuous
with respect to the product measure $\alpha^n$.
Moreover,
$$
\frac{d\tilde{\alpha}_n}{d\alpha^n}((x_1,y_1),\dots,(x_n,y_n))\, =\, \frac{1}{Z_n(\alpha,\beta)}\, \exp\Big(-\beta H_n((x_1,y_1),\dots,(x_n,y_n))\Big)\, .
$$
Here the Hamiltonian is 
$$
H_n((x_1,y_1),\dots,(x_n,y_n))\, =\, \frac{1}{n-1} \sum_{i=1}^{n-1} \sum_{j=i+1}^{n} h(x_i-x_j,y_i-y_j)\, .
$$
This leads us to define a conditional Hamiltonian for the single point $(x,y)$ substituted in for $(x_k,y_k)$ in the configuration $((x_1,y_1),\dots,(x_n,y_n))$:
$$
H_{n,k}\big((x,y);(x_1,y_1),\dots,(x_n,y_n)\big)\, 
=\, \frac{1}{n-1} \sum_{\substack{j=1\\j\neq k}}^n h_n(x-x_j,y-y_j)\, .
$$
With this, we define a measure $\tilde{\alpha}_{n,k}\big(\cdot;(x_1,y_1),\dots,(x_n,y_n)\big)$, which is absolutely continuous
with respect to $\alpha$, and such that
$$
\frac{d\tilde{\alpha}_{n,k}\big(\cdot;(x_1,y_1),\dots,(x_n,y_n)\big)}{d\alpha}(x,y)\,
=\, \frac{1}
{Z_{n,k}\big(\alpha,\beta;(x_1,y_1),\dots,(x_n,y_n)\big)}e^{-\beta H_{n,k}((x,y);(x_1,y_1),\dots,(x_n,y_n))}\, .
$$
The normalization is 
$$
Z_{n,k}\big(\alpha,\beta;(x_1,y_1),\dots,(x_n,y_n)\big)\, =\, \int_{\R^2} e^{-\beta H_{n,k}((x,y);(x_1,y_1),\dots,(x_n,y_n))}\, d\alpha(x,y)\, .
$$
To see that this is the desired regular conditional probability distribution,
note that in the product
$$
\frac{d\tilde{\alpha}_{n,k}\big(\cdot;(x_1,y_1),\dots,(x_n,y_n)\big)}{d\alpha}(x,y)\, \frac{d\tilde{\alpha}_n}{d\alpha^n}((x_1,y_1),\dots,(x_n,y_n))
$$
we have the product of two factors:
$$
\frac{1}
{Z_{n,k}\big(\alpha,\beta;(x_1,y_1),\dots,(x_n,y_n)\big)}e^{-\beta H_{n,k}((x,y);(x_1,y_1),\dots,(x_n,y_n))}
$$
and
$$
\frac{1}{Z_n(\alpha,\beta)}\, \exp\Big(-\beta H_n((x_1,y_1),\dots,(x_n,y_n))\Big)\, .
$$
The first factor does not depend on $(x_k,y_k)$.
The second factor does depend on it, but integrating against $d\alpha(x_k,y_k)$ gives,
$$
\int_{\R^2} \frac{ e^{-\beta H_n((x_1,y_1),\dots,(x_n,y_n))}}{Z_n(\alpha,\beta)}\, d\alpha(x_k,y_k)\,
=\, \frac{Z_{n,k}\big(\alpha,\beta;(x_1,y_1),\dots,(x_n,y_n)}{Z_n(\alpha,\beta)} e^{-\beta H_{n,k}'((x_1,y_1),\dots,(x_n,y_n))} 
$$
where
$$
H_{n,k}'((x_1,y_1),\dots,(x_n,y_n))\, =\, \frac{1}{n-1} \sum_{\substack{i=1\\i\neq k}}^{n-1} \sum_{\substack{j=i+1\\j\neq k}}^{n} h(x_i-x_j,y_i-y_j)\, ,
$$
and we have
\begin{multline*}
H_{n,k}'((x_1,y_1),\dots,(x_n,y_n)) + H_{n,k}\big((x,y);(x_1,y_1),\dots,(x_n,y_n)\big)\\
=\, H_{n}\big((x_1,y_1),\dots,(x_{k-1},y_{k-1}),(x,y),(x_{k+1},y_{k+1}),\dots,(x_n,y_n)\big)\, .
\end{multline*}
Therefore,
$$
\int_{\R^2} \frac{d\tilde{\alpha}_{n,k}\big(\cdot;(x_1,y_1),\dots,(x_n,y_n)\big)}{d\alpha}(x,y)\, \frac{d\tilde{\alpha}_n}{d\alpha^n}((x_1,y_1),\dots,(x_n,y_n))\, d\alpha(x_k,y_k)
$$
equals
$$
\frac{d\tilde{\alpha}_n}{d\alpha^n}\big((x_1,y_1),\dots,(x_{k-1},y_{k-1}),(x,y),(x_{k+1},y_{k+1}),\dots,(x_n,y_n)\big)\, .
$$
This implies condition 3 in Definition \ref{def:rcpd}.
Conditions 1 and 2 are true because of the joint measurability of the density, which just depends on the Hamiltonian.

Note that for any pair of points $(x,y)$, $(x',y')$, we have
\begin{equation}
\label{ineq:variation}
\left|H_{n,k}\big((x,y);(x_1,y_1),\dots,(x_n,y_n)\big) -
H_{n,k}\big((x',y');(x_1,y_1),\dots,(x_n,y_n)\big)\right|\, \leq\, 1\, ,
\end{equation}
because $|h(x-x_j,y-y_j) - h(x'-x_j,y'-y_j)|$ is either $0$ or $1$ for each $j$, and $H_{n,k}$ is a sum of $n-1$ such terms, then divided by $n-1$.
We may write
$$
\left(\frac{d\tilde{\alpha}_{n,k}\big(\cdot;(x_1,y_1),\dots,(x_n,y_n)\big)}{d\alpha}(x,y)\right)^{-1}\,
=\, Z_{n,k}\big(\alpha,\beta;(x_1,y_1),\dots,(x_n,y_n)\big)e^{\beta H_{n,k}((x,y);(x_1,y_1),\dots,(x_n,y_n))}
$$
as an integral
$$
\int_{\R^2} e^{\beta \left[H_{n,k}\big((x,y);(x_1,y_1),\dots,(x_n,y_n)\big) -
H_{n,k}\big((x',y');(x_1,y_1),\dots,(x_n,y_n)\big)\right]}\, d\alpha(x',y')\, .
$$
Therefore, the inequality (\ref{ineq:variation}) implies that
$$
e^{-|\beta|}\, \leq\, \left(\frac{d\tilde{\alpha}_{n,k}\big(\cdot;(x_1,y_1),\dots,(x_n,y_n)\big)}{d\alpha}(x,y)\right)^{-1}\,
\leq\, e^{|\beta|}\, .
$$
Of course, this implies the same bounds for the reciprocal. For all $(x,y) \in \R^2$,
$$
e^{-|\beta|}\, \leq\, \frac{d\tilde{\alpha}_{n,k}\big(\cdot;(x_1,y_1),\dots,(x_n,y_n)\big)}{d\alpha}(x,y)\,
\leq\, e^{|\beta|}\, .
$$
So, taking $p = e^{-|\beta|}$,
this means that the hypotheses of Lemma \ref{lem:CouplingGeneral} are satisfied: $\tilde{\alpha}_n$
has the ``$p$-resampling'' property relative to the measure $\alpha$.
Hence, we conclude that Lemma \ref{lem:CouplingAboveBelow} is true.

\section*{Acknowledgements}

We are very grateful to Janko Gravner for helpful suggestions, including directing us to reference \cite{DeuschelZeitouni}.
S.S\ is also grateful for advice from Bruno Nachtergaele, and he
is grateful for the warm hospitality of the Erwin Schr\"odinger Institute
where part of this research occurred.

\end{document}